\documentclass{amsproc}

\usepackage{amssymb,amsmath,amsfonts,stmaryrd,mathrsfs,amscd}
\usepackage[sans]{dsfont}
\usepackage{fouridx,chemarrow,eufrak}
\usepackage[dvips,cmtip,all]{xy}
\usepackage{setspace,paralist}
\usepackage[linkcolor=black]{hyperref}
\usepackage{lineno,color}
\usepackage{tikz,mathabx} 
\usetikzlibrary{matrix,arrows,decorations.pathmorphing}

\newtheorem{theorem}{Theorem}[section]
\newtheorem{lemma}[theorem]{Lemma}

\theoremstyle{definition}
\newtheorem{definition}[theorem]{Definition}

\newtheorem{proposition}[theorem]{Proposition}
\newtheorem{corollary}[theorem]{Corollary}
\newtheorem{example}[theorem]{Example}

\numberwithin{equation}{section}

\theoremstyle{remark}
\newtheorem{remark}[theorem]{Remark}

\numberwithin{equation}{section}

\def\depth{\operatorname{depth}}

\def\eh{\operatorname{H}}
\def\edim{\operatorname{edim}}

\def\im{\operatorname{image}}

\def\tor{\operatorname{Tor}}
\def\eff{\operatorname{\Phi}}
\def\rmod{\operatorname{\textbf{R-Mod}}}
\def\flatdim{\operatorname{flat\ dim_\mathit{R}}}
\def\Fin{\operatorname{\varphi_{\varstar}^{\mathit{n}}}}

\def\Ef{\operatorname{\mathit{f}_{\varstar}}}

\begin{document}

\title{Kunz' Regularity Criterion via algebraic entropy}

\author{Mahdi Majidi-Zolbanin}
\address{Department of Mathematics, LaGuardia Community College of the City University of New York, 31-10 Thomson Avenue, Long Island City, NY 11101}
\email{mmajidi-zolbanin@lagcc.cuny.edu}

\author{Nikita Miasnikov}
\address{Department of Mathematics,The Graduate Center of the City University of New York, 365 Fifth Avenue, New York, NY 10016}
\email{n5k5t5@gmail.com}

\author{Lucien Szpiro}
\address{Department of Mathematics,The Graduate Center of the City University of New York, 365 Fifth Avenue, New York, NY 10016}
\email{lszpiro@gc.cuny.edu}
\thanks{The second and third authors received funding from the NSF Grants DMS-0854746 and DMS-0739346.}


\date{October 5, 2011.}

\dedicatory{}


\begin{abstract}
In~\cite{we2} we introduced and studied a notion of algebraic entropy. In this paper we will give an application of algebraic entropy in proving Kunz' Regularity Criterion for all contracting self-maps of finite length of Noetherian local rings in arbitrary characteristic. Some conditions of Kunz' Criterion have already been extended to the general case in~\cite{AvrIynMill06}, using different methods.
\end{abstract}

\maketitle

\section{Introduction and notations}\label{Introduction}
Let \((R,\mathfrak{m})\) be a Noetherian local ring of positive prime characteristic $p$ and of dimension $d$. In~\cite{Kunz1} Kunz showed that the following conditions are equivalent:\smallskip

\begin{compactenum}
\item[\textbf{a})] $R$ is regular.
\item[\textbf{b})] The Frobenius endomorphism $f:R\rightarrow R$ is flat.
\item[\textbf{c})] The length $\ell_R\big(R/f(\mathfrak{m})R\big)$ is equal to $p^d$.
\item[\textbf{d})] The length $\ell_R\big(R/f^n(\mathfrak{m})R\big)$ is equal to $p^{nd}$ for some $n\in\mathbb{N}$.
\end{compactenum}\vspace{0.2cm}

Later, in~\cite{Rodic}, Rodicio showed these conditions are also equivalent to\smallskip

\begin{compactenum}
\item[\textbf{b$^\prime$})] $\Ef R$ has finite flat dimension over $R$.
\end{compactenum}\vspace{0.2cm}

At first glance, Kunz' conditions \textbf{c}) and \textbf{d}) appear to be stated in terms of the characteristic $p$ of the ring and one may not expect to be able to extend them, or even state them in arbitrary characteristic. Our goal in this paper is to use the notion of \emph{algebraic entropy}, that was introduced and studied in~\cite{we2}, to make sense of Kunz' conditions \textbf{c}) and \textbf{d}) for all self-maps of finite length (see Definition~\ref{deflambda}) of Noetherian local rings in any characteristic. We will then show that with this new interpretation, all conditions in Kunz' result are still equivalent. We should note that Avramov, Iyengar and Miller have already extended the equivalence of conditions \textbf{a}) and \textbf{b}) of Kunz and \textbf{b$^\prime$}) of Rodicio to all \emph{contracting} local self-maps of Noetherian local rings in any characteristic, in~\cite{AvrIynMill06}. In order to state their result, we shall first recall the definition of a contracting self-map.
\begin{definition}[\cite{AvrIynMill06}, p.~80]\label{contracting}
A local self-map \(\varphi\) of a local ring \((R,\mathfrak{m})\) is said to be \emph{contracting}, if for each element \(x\in\mathfrak{m}\) the sequence \(\{\varphi^n(x)\}_{i\geq1}\) converges to \(0\) in the \(\mathfrak{m}\)-adic topology of \(R\).
\end{definition}
\begin{remark}[\cite{AvrIynMill06}, Lemma~12.1.4]\label{alternative}
A local self-map \(\varphi\) of a Noetherian local ring \((R,\mathfrak{m})\) is contracting if and only if \(\varphi^{\edim(R)}(\mathfrak{m})R\subset\mathfrak{m}^2\), where \(\edim(R)\) is the embedding dimension of \(R\).
\end{remark}
Here we state (part of) the result obtained by Avramov, Iyengar and Miller. 
\begin{theorem}[{\cite[Theorem~13.3]{AvrIynMill06}}]\label{theirs}
Let $(R,\mathfrak{m})$ be a Noetherian local ring, and let $\varphi$ be a contracting local self-map of $R$. Then the following conditions are equivalent:\smallskip

\begin{compactenum}
\item[\emph{\textbf{1})}] $R$ is regular.
\item[\emph{\textbf{2})}] $\flatdim \Fin R=\dim R/(\varphi^n(\mathfrak{m})R)$ for all integers $n\geq1$.
\item[\emph{\textbf{3})}] $\flatdim \Fin R<\infty$ for some integer $n\geq1$.
\end{compactenum}
\end{theorem}
In order to state our main result, we shall first recall a number of definitions, terminology and notation, that will be used throughout this paper.
\begin{definition}[{\cite[Definition~2.1]{we2}}]\label{deflambda}
Let \(f:(R,\mathfrak{m})\rightarrow(S,\mathfrak{n})\) be a homomorphism of Noetherian local rings. We say $f$ is \emph{of finite length}, if it is local and \(f(\mathfrak{m})S\) is \(\mathfrak{n}\)-primary. In this case we define the length of $f$ as \(\lambda(f):=\ell_S\big(S/f(\mathfrak{m})S\big)\). It is clear that \(\lambda(f)\in[1,\infty)\).
\end{definition}
\begin{remark}
For local homomorphisms of Noetherian local rings:
\begin{center} 
Finite $\Rightarrow$ Integral $\Rightarrow$ Finite length.
\end{center} 
The converse need not be true.
\end{remark}
In~\cite{we2} we introduced and studied a notion of algebraic entropy for self-maps of finite length of Noetherian local rings. Here we recall its definition.
\begin{definition}
Let \((R,\mathfrak{m})\) be a Noetherian local ring, let \(\varphi\) be a self-map of finite length of \(R\). We define the \emph{algebraic entropy} of \(\varphi\) as 
\begin{equation}\label{entropy limit}
h_{\mathrm{alg}}(\varphi,R):=\lim_{n\rightarrow\infty}\frac{\log\lambda(\varphi^n)}{n}.
\end{equation}
\end{definition}
\begin{remark}\label{infimum}
The limit in Equation~\ref{entropy limit} always exists and is finite. Furthermore, the sequence $\{\log\lambda(\varphi^n)/n\}$ always converges to its infimum. See~\cite[Theorem~3.2]{we2}.
\end{remark}
\begin{example}
In a Noetherian local ring $R$ of positive prime characteristic $p$ and dimension $d$, the algebraic entropy of the Frobenius endomorphism is equal to $d\cdot\log p$. See~\cite[Example~3.6]{we2}.
\end{example}
Our main result in this paper is the following:
\begin{theorem}\label{main}
Let $\varphi$ be a self-map of finite length of a Noetherian local ring $(R,\mathfrak{m})$ of arbitrary characteristic. Define $q(\varphi):=\exp(h_{\mathrm{alg}}(\varphi,R))$. Consider the following conditions:\smallskip

\begin{compactenum}
\item[\emph{\textbf{a})}] $R$ is regular.
\item[\emph{\textbf{b})}] $\varphi:R\rightarrow R$ is flat.
\item[\emph{\textbf{c})}] $\lambda(\varphi)=q(\varphi)$.
\item[\emph{\textbf{d})}] $\lambda(\varphi^n)=q(\varphi)^n$ for some $n\in\mathbb{N}$.
\end{compactenum}\vspace{0.2cm}

\noindent Then \emph{\textbf{a})} $\Rightarrow$ \emph{\textbf{b})} $\Rightarrow$ \emph{\textbf{c})} $\Rightarrow$ \emph{\textbf{d})}. If in addition $\varphi$ is contracting, then \emph{\textbf{d})} $\Rightarrow$ \emph{\textbf{b})} $\Rightarrow$ \emph{\textbf{a})}. That is, in this case all conditions are equivalent.
\end{theorem}In our proof of Theorem~\ref{main}, we will use a proof of Herzog originally written in~\cite[Satz~3.1]{Herzog} for the Frobenius endomorphism, to prove the implication \textbf{b}) $\Rightarrow$ \textbf{a}) in the general setting. This part of our proof, however, is not new and has already appeared in~\cite[Lemma~3]{BrunGubel}.\par
From Theorem~\ref{main} we quickly see that the Hilbert-Kunz multiplicity of a regular local ring with respect to an arbitrary self-map of finite length, as defined in~\cite[Definition~3.7]{we2}, is $1$. This is well-known in the case of the Frobenius endomorphism.\bigskip

\noindent \textbf{Notations.} All rings in this paper are assumed to be Noetherian, commutative and with identity element. By a self-map of a ring we mean an endomorphism of that ring. For a self-map \(\varphi\) of a ring we will write \(\varphi^n\) for the \(n\)-fold composition of \(\varphi\) with itself. Given a ring homomorphism \(f:R\rightarrow S\) and an \(S\)-module \(N\), we will denote by \(\Ef N\) the \(R\)-module obtained by restriction of scalars. That is, \(\Ef N\) is the \(R\)-module whose underlying abelian group is \(N\) and whose \(R\)-module structure is given by \(r\cdot x=f(r)\:x\), for \(r\in R\) and \(x\in\Ef N\). This notation is consistent with the one used in~\cite{Bourb}. If \(M\) is an \(R\)-module of finite length, we will denote its length by \(\ell_R(M)\). 
\section{Preliminaries}

\begin{proposition}[{\cite[Proposition~2.9]{we2}}]\label{flatchange}
Assume that \(f:(R,\mathfrak{m})\rightarrow(S,\mathfrak{n})\) is a homomorphism of finite length of Noetherian local rings, and let \(M\) be an \(R\)-module of finite length. Then\smallskip

\begin{compactenum}
\item[\textbf{a})] \(M\otimes_RS\) is an \(S\)-module of finite length.
\item[\textbf{b})] In general \(\ell_S(M\otimes_R S)\leq\lambda(f)\cdot\ell_R(M)\).
\item[\textbf{c})] If in addition \(f\) is \emph{flat}, then \(\ell_S(M\otimes_R S)=\lambda(f)\cdot\ell_R(M)\).
\end{compactenum}
\end{proposition}
\begin{proposition}[{\cite[Corollary~2.10]{we2}}]\label{magic}
Assume that \(f:(R,\mathfrak{m})\rightarrow(S,\mathfrak{n})\) and \(g:(S,\mathfrak{n})\rightarrow(T,\mathfrak{p})\) are homomorphisms of finite length of Noetherian local rings. Then with notation of Definition~\ref{deflambda},\smallskip 

\begin{compactenum}
\item[\textbf{a})] In general \(\lambda(g)\leq\lambda(g\circ f)\leq\lambda(g)\cdot\lambda(f)\).
\item[\textbf{b})] If in addition \(g\) is \emph{flat}, then \(\lambda(g\circ f)=\lambda(g)\cdot\lambda(f)\).
\end{compactenum}
\end{proposition}
\begin{definition}\label{functor}
Let $R$ be a Noetherian local ring, and let $\varphi$ be a self-map of $R$. Let $\rmod$ be the category of $R$-modules. For every $n\in\mathbb{N}$ we define a functor $\eff^n:\rmod\rightarrow\rmod$ as follows: if $M\in\rmod$, then
\begin{equation}\label{functoriel}
\eff^n(M):=M\otimes_R\Fin R,
\end{equation}
where the \(R\)-module structure of \(\eff^n(M)\) is defined to be 
$$r\cdot x=\sum m_i\otimes r\cdot r_i,\ \ \mathrm{if}\ \ x=\sum m_i\otimes r_i\in\eff^n(M)\ \ \mathrm{and}\ \ r\in R.$$
\end{definition}
\begin{proposition}\label{properties}
Let $R$ be a Noetherian local ring, and let $\varphi$ be a self-map of $R$. The functors $\eff^n$ ($n\in\mathbb{N}$) have the following properties:\smallskip

\begin{compactenum}
\item[\textbf{a})] $\eff^n$ is a right-exact functor.
\item[\textbf{b})] If $R^s$ is a finitely generated free module, then $\eff^n(R^s)\cong R^s$.
\item[\textbf{c})] Let $R^s\stackrel{\alpha}{\rightarrow}R^t$ be a map of finitely generated free $R$-modules. Choose bases $\mathscr{B}_s$ and $\mathscr{B}_t$ for $R^s$ and $R^t$,  and let $(a_{ij})$ be the matrix representation of $\alpha$ in these bases. Then the matrix representation of $\eff^n(\alpha)$ in the bases of $\eff^n(R^s)$ and $\eff^n(R^t)$ obtained from $\mathscr{B}_s$ and $\mathscr{B}_t$ by applying the isomorphisms of \textbf{b}) is $(\varphi^n(a_{ij}))$.
\item[\textbf{d})] If $\mathfrak{a}$ is an ideal of $R$, then $\eff^n(R/\mathfrak{a})\cong R/\varphi^n(\mathfrak{a})R$, as $R$-modules. 
\item[\emph{\textbf{e})}] If $M$ is an $R$-module of finite length, then $\eff^n(M)$ is an $R$-module of finite length, and $\ell_R(\eff^n(M))\leq\ell_R(M)\cdot\lambda(\varphi^n)$.
\end{compactenum}
\end{proposition}
\begin{proof}
\textbf{a}) The functor $\;\cdot\;\otimes_R\Fin R$ is right-exact, and changing the module structure will not affect the kernels and images. Thus, $\eff^n$ is right-exact.\par
\textbf{b}) Let $\{e_1,\ldots,e_s\}$ be a basis for $R^s$. Consider the maps $\sigma:R^s\rightarrow\eff^n(R^s)$ and $\delta:\eff^n(R^s)\rightarrow R^s$ defined as following:
\begin{eqnarray*}
\sum_{i=1}^sr_ie_i&\stackrel{\sigma}{\mapsto}&\sum_{i=1}^se_i\otimes r_i\\
\sum_j(\sum_{i=1}^sr_{ij}e_i)\otimes a_j&\stackrel{\delta}{\mapsto}&\sum_{i=1}^s(\sum_j\varphi^n(r_{ij})a_j)e_i.
\end{eqnarray*}
One can check that these maps define $R$-module homomorphisms, and $\sigma\circ\delta$ and $\delta\circ\sigma$ are both equal to identity. Thus, $\sigma$ and $\delta$ are isomorphisms.\par
\textbf{c}) Suppose $\mathscr{B}_s=\{e_1,\ldots,e_s\}$ and $\mathscr{B}_t=\{f_1,\ldots,f_t\}$. From the proof of \textbf{b}) we can see that $\{e_1\otimes1,\ldots,e_s\otimes1\}$ and $\{f_1\otimes1,\ldots,f_t\otimes1\}$ are bases of $\eff^n(R^s)$ and $\eff^n(R^t)$, obtained from $\mathscr{B}_s$ and $\mathscr{B}_t$ by applying the isomorphisms of \textbf{b}). We have
\begin{eqnarray*}
\eff^n(\alpha)(e_k\otimes1)&=&\sum_{i=1}^t \alpha(e_k)\otimes1\\
&=&\sum_{i=1}^t a_{ik}f_i\otimes1\\
&=&\sum_{i=1}^t f_i\otimes\varphi^n(a_{ik})\\
&=&\varphi^n(a_{ik})\sum_{i=1}^t f_i\otimes1.
\end{eqnarray*}
The result immediately follows from this equation.\par
\textbf{d}) Consider a minimal presentation of $R/\mathfrak{a}$
\begin{equation}\label{quoseqgen}
R^t\stackrel{\alpha}\rightarrow R\rightarrow R/\mathfrak{a}\rightarrow0.
\end{equation}
Applying the functor $\eff^n$ to this sequence, by \textbf{a}) we obtain an exact sequence
$$\eff^n(R^t)\rightarrow \eff^n(R)\rightarrow \eff^n(R/\mathfrak{a})\rightarrow0.$$ Let $\mathscr{B}_t=\{f_1,\ldots,f_t\}$ be a basis of $R^t$ and consider the basis $\mathscr{B}_1=\{1\}$ of $R$. As mentioned above, $\{e_1\otimes1,\ldots,e_s\otimes1\}$ and $\{1\otimes1\}$ are bases of $\eff^n(R^t)$ and $\eff^n(R)$, obtained from $\mathscr{B}_t$ and $\mathscr{B}_1$ by applying the isomorphisms of \textbf{b}). Let $(a_{1j})$ be the matrix representation of $\alpha$ with respect to $\mathscr{B}_t$ and $\mathscr{B}_1$. Then by \textbf{c}) the matrix representation of $\eff^n(\alpha)$ with respect to the above-mentioned bases of $\eff^n(R^t)$ and $\eff^n(R)$ is $(\varphi^n(a_{1j}))$. From Exact Sequence~\ref{quoseqgen} we see that $a_{11},\ldots,a_{1t}$ are generators of the ideal $\mathfrak{a}$. Hence, $\varphi^n(a_{11}),\ldots,\varphi^n(a_{1t})$ are generators of $\varphi^n(\mathfrak{a})R$. We can quickly see that the following diagram, in which vertical arrows are isomorphisms of \textbf{b}), is commutative
\[\begin{tikzpicture} 
\matrix (m) [matrix of math nodes, row sep=3.2em, column sep=2.6em, text height=2ex, text depth=0.25ex] { \eff^n(R^t) & \eff^n(R) & \eff^n(R/\mathfrak{a}) & 0\\ R^t & R & R/\varphi^n(\mathfrak{a})R & 0\\}; 
\path[->, font=\scriptsize, >=angle 60]
(m-1-1) edge node[auto] {$ (\varphi^n(a_{1j})) $} (m-1-2) 
             edge node[right] {$ \wr $} (m-2-1)
(m-1-2) edge (m-1-3) 
(m-1-3) edge (m-1-4)
(m-2-1) edge node[auto] {$ (\varphi^n(a_{1j})) $} (m-2-2)
(m-2-2) edge (m-2-3) 
(m-2-3) edge (m-2-4)
(m-1-2) edge node[right] {$ \wr $} (m-2-2);
\end{tikzpicture}\]
This diagram induces an $R$-module isomorphism $\eff^n(R/\mathfrak{a})\stackrel{\sim}{\rightarrow} R/\varphi^n(\mathfrak{a})R$.\par
\textbf{e}) These are restatements, in terms of $\eff^n$, of parts \textbf{a}) and \textbf{b}) of Proposition~\ref{flatchange}.
\end{proof}
Below, we list a number of other results, that we will need in the proof of Theorem~\ref{main}.
\begin{lemma}[{\cite[Lemma~3.2]{Herzog}}]\label{Herzoglemma}
Let $(R,\mathfrak{m})$ be a Noetherian local ring, and let $M$ be a finitely generated $R$-module. Consider an ideal $\mathfrak{b}\subseteq\mathfrak{m}$ of $R$. Then there exists an integer $\mu_0\geq0$ such that $\depth(\mathfrak{m},\mathfrak{b}^{\mu}M)>0$ for all $\mu\geq \mu_0$. 
\end{lemma}
\begin{remark}
Lemma~\ref{Herzoglemma} must be used together with the standard convention, that the depth of the zero module is $\infty$ (see, for example,~\cite[p.~291]{Huneke}). For instance, if $M$ is an $R$-module of finite length, then for $\mu\gg0$ we have $\mathfrak{m}^\mu M=(0)$.
\end{remark}
\begin{proposition}[{\cite[Chap.~10, \S~1, Proposition~1]{Bourb2}}]\label{depthl}
Let $R$ be a ring , $\mathfrak{a}$ an ideal of $R$, and let $0\rightarrow M^\prime\rightarrow M\rightarrow M^{\prime\prime}\rightarrow0$ be an exact sequence of $R$-modules. Define
$d^\prime=\depth(\mathfrak{a},M^\prime)$, $d=\depth(\mathfrak{a},M)$, and $d^{\prime\prime}=\depth(\mathfrak{a},M^{\prime\prime})$. Then we will have one of the following mutually exclusive possibilities:
$$d^\prime=d\leq d^{\prime\prime},\ \ \ d=d^{\prime\prime}<d^\prime,\ \ \ d^{\prime\prime}=d^\prime-1<d.$$
\end{proposition}
The next lemma is Nagata's Flatness Criterion. A proof can be found in~\cite[Chap.~II, Theorem~19.1, p.~64]{Nagata2}. See also~\cite[Ex.~22.1, p.~178]{Matsumura2}.
\begin{lemma}[Nagata]\label{Nagata}
Let $(R,\mathfrak{m})$ and $(S,\mathfrak{n})$ be Noetherian local rings. Suppose $R\subset S$, and assume that $\mathfrak{m}S$ is an $\mathfrak{n}$-primary ideal. Then $S$ is flat over $R$, if and only if for every $\mathfrak{m}$-primary ideal $\mathfrak{q}$ of $R$,
\begin{equation}\label{transition}
\ell_R(R/\mathfrak{q})\cdot\ell_S(S/\mathfrak{m}S)=\ell_S(S/\mathfrak{q}S).
\end{equation}
\end{lemma}
\section{Kunz Regularity Criterion}
In this section we will present the proof of our main result, Theorem~\ref{main}. We first need to establish a lemma, that we will need in our proof of this Theorem.
\begin{lemma}\label{forinj}
Let $\varphi$ be a self-map of finite length of a Noetherian local ring $(R,\mathfrak{m})$, and let $\mathfrak{a}$ be a $\varphi$-invariant ideal of $R$, i.e., $\varphi(\mathfrak{a})R\subset\mathfrak{a}$. Let $\overline{\varphi}$ be the self-map of $R/\mathfrak{a}$ induced by $\varphi$. Let $q(\varphi)$ be as defined in Theorem~\ref{main}.\smallskip 

\begin{compactenum}
\item[\emph{\textbf{i})}] If $\lambda(\varphi^n)=q(\varphi)^n$ for some $n\in\mathbb{N}$, then $\lambda(\varphi^{nt})=q(\varphi)^{nt}$ for all $t\in\mathbb{N}$.
\item[\emph{\textbf{ii})}] If in addition to the assumption in \emph{\textbf{i})}, we have $h_{\mathrm{alg}}(\overline{\varphi},R/\mathfrak{a})=h_{\mathrm{alg}}(\varphi,R)$ and if $\varphi$ is contracting, then $\mathfrak{a}=(0)$.
\end{compactenum}
\end{lemma}
\begin{proof}
\textbf{i}): Let $t\in\mathbb{N}$. As noted in Remark~\ref{infimum}, the sequence $\{\log\lambda(\varphi^{nt})/(nt)\}$ converges to its infimum as $t\rightarrow\infty$. Hence, $$h_{\mathrm{alg}}(\varphi,R)\leq\log\lambda(\varphi^{nt})/(nt).$$ From this inequality we quickly obtain $q(\varphi)^{nt}\leq\lambda(\varphi^{nt})$. On the other hand, by Proposition~\ref{magic} $\lambda(\varphi^{nt})\leq\lambda(\varphi^n)^t$. Using assumption \textbf{i}) and the previous inequalities we obtain
\begin{eqnarray*}
q(\varphi)^{nt}&\leq&\lambda(\varphi^{nt})\nonumber\\
&\leq&\lambda(\varphi^n)^t\nonumber\\
&=&q(\varphi)^{nt}\nonumber.
\end{eqnarray*}
Hence, $\lambda(\varphi^{nt})=q(\varphi)^{nt}$ for all $t\in\mathbb{N}$.\par 
\textbf{ii}): Similarly, 
\begin{eqnarray}\label{second}
q(\overline{\varphi})^{nt}&\leq&\lambda(\overline{\varphi}^{nt})\nonumber\\
&\leq&\lambda(\varphi^{nt})\\
&=&q(\varphi)^{nt}\nonumber.
\end{eqnarray}
From assumption \textbf{ii}) it follows that $q(\overline{\varphi})=q(\varphi)$. Then from Equation~\ref{second} we conclude $\lambda(\overline{\varphi}^{nt})=\lambda(\varphi^{nt})$ for all $t\in\mathbb{N}$. Since $\lambda(\overline{\varphi}^{nt})=\ell_R(R/[\varphi^{nt}(\mathfrak{m})R+\mathfrak{a}])$ (see~\cite[Proposition~2.8]{we2}), we obtain
\begin{equation}\label{lengths}
\ell_R(R/[\varphi^{nt}(\mathfrak{m})R+\mathfrak{a}])=\ell_R(R/\varphi^{nt}(\mathfrak{m})R),\ \ \ \forall t\in\mathbb{N}.
\end{equation}
The surjection 
$$R/\varphi^{nt}(\mathfrak{m})R\rightarrow R/[\varphi^{nt}(\mathfrak{m})R+\mathfrak{a}])\rightarrow0,$$
and Equation~\ref{lengths} then show 
$$R/[\varphi^{nt}(\mathfrak{m})R+\mathfrak{a}]=R/\varphi^{nt}(\mathfrak{m})R,\ \ \ \forall t\in\mathbb{N}.$$
Hence,
$$\mathfrak{a}\subset\bigcap_{t\in\mathbb{N}}\varphi^{nt}(\mathfrak{m})R=(0),$$
where the last equality follows from Remark~\ref{alternative}, because $\varphi$ is, by assumption, contracting.
\end{proof}
Now we are ready to present our proof of Theorem~\ref{main}:
\begin{proof}
\textbf{a}) $\Rightarrow$ \textbf{b}): To say that $\varphi$ is of finite length means $\dim R/\varphi(\mathfrak{m})R=0$. Hence, the following equation holds:
$$\dim R = \dim R + \dim R/\varphi(\mathfrak{m})R.$$ 
Since $R$ is regular, the result easily follows from, for example,~\cite[Theorem~23.1]{Matsumura2}.\par
\textbf{b}) $\Rightarrow$ \textbf{c}): This follows from Proposition~\ref{magic}. Since $\varphi$ is assumed to be flat, by that proposition for all $n\in\mathbb{N}$ we obtain
$$\lambda(\varphi^n)=\lambda(\varphi)^n.$$
Thus, by definition of algebraic entropy
\begin{eqnarray*}
h_{\mathrm{alg}}(\varphi,R)&=&\lim_{n\rightarrow\infty}(1/n)\cdot\log\lambda(\varphi^n)\nonumber\\
&=&\lim_{n\rightarrow\infty}(1/n)\cdot\log\lambda(\varphi)^n\nonumber\\
&=&\log\lambda(\varphi).\nonumber
\end{eqnarray*}
This means $\lambda(\varphi)=q(\varphi)$.\par
\textbf{c}) $\Rightarrow$ \textbf{d}): This is clear.\par
\textbf{b}) $\Rightarrow$ \textbf{a}): (This is essentially Herzog's proof in~\cite[Satz~3.1]{Herzog}. We will re-write it for an arbitrary self-map here. See also~\cite[Lemma~3]{BrunGubel}.) To show that $R$ is regular, it suffices to show all finitely generated $R$-modules have finite projective dimension. So let $M$ be a finitely generated $R$-module. Suppose $M$ were of infinite projective dimension. Then consider a minimal (infinite) free resolution of $M$
$$L_{\bullet}\rightarrow M\rightarrow0.$$
Let $s:=\depth(\mathfrak{m},R)$, and take an $R$-regular sequence of elements $\{x_1,\ldots,x_s\}$ in $\mathfrak{m}$. Write $\mathfrak{a}$ for the ideal generated by this regular sequence. (If $s=0$, take $\mathfrak{a}=(0)$.) For every $n\in\mathbb{N}$ we set
$$C^n_{\bullet}:=\eff^n(L_{\bullet})\otimes_RR/\mathfrak{a}\ \ \ \ \mathrm{and}\ \ \ \ B^n_i:=\im(C^n_{i+1}\rightarrow C^n_i).$$
Using Proposition~\ref{properties}-\textbf{b}, we quickly see that $C^n_i\cong L_i/\mathfrak{a}L_i$. This shows that $C^n_i$ is independent of $n$, and that $C^n_i$ is a nonzero finitely generated module of depth zero for all $i$. Using Proposition~\ref{properties}-\textbf{c}, we can see that $B^n_i\subseteq\varphi^n(\mathfrak{m})C_i^n$. Applying Lemma~\ref{Herzoglemma}, let $\mu_{i_0}$ be such that $\depth(\mathfrak{m},\mathfrak{m}^{\mu_{i_0}} C^n_i)>0$. Since $\varphi$ is contracting by assumption, from Remark~\ref{alternative} it easily follows that if $n$ is large enough, then $\varphi^{n}(\mathfrak{m})R\subseteq\mathfrak{m}^{\mu_{i_0}}$ and in that case, $B^n_i\subseteq\varphi^n(\mathfrak{m})C_i^n\subseteq \mathfrak{m}^{\mu_{i_0}}C_i^n$. This shows that $\depth(\mathfrak{m},B_i^n)>0$ for large $n$.\par On the other hand, since $\varphi$ is flat by assumption, $\eff^n(L_{\bullet})$ is exact. Thus, by parts \textbf{a}), \textbf{b}), and \textbf{c}) of Proposition~\ref{properties} $$\eff^n(L_{\bullet})\rightarrow\eff^n(M)\rightarrow0$$ is a minial (infinite) free resolution of $\eff^n(M)$. Hence $$\eh_i(C^n_{\bullet})=\tor_i^R(\eff^n(M),R/\mathfrak{a})=0,\ \ \mathrm{for}\ \ i>s.$$
This shows that if $i>s$, then the sequences
\begin{equation}\label{conseq}
0\rightarrow B^n_{i+1}\rightarrow C^n_{i+1}\rightarrow B^n_i\rightarrow0
\end{equation} 
are exact for all $n\in\mathbb{N}$. Take $i=s+1$ in Sequence~\ref{conseq}, for instance. By the above argument, if we take $n$ large enough, we will obtain $\depth(\mathfrak{m},B_{s+1}^n)>0$ and $\depth(\mathfrak{m},B_{s+2}^n)>0$, while $\depth(\mathfrak{m},C^n_{s+2})=0$. This contradicts Proposition~\ref{depthl}. This contradiction shows that the projective dimension of $M$ must be finite.\par
\textbf{d}) $\Rightarrow$ \textbf{b}): We will use Nagata's Flatness Criterion (Lemma~\ref{Nagata}) to show that $\varphi^n$ is flat. To apply this criterion, we first need to show that $\varphi$ is injective.\par Clearly, $\ker\varphi$ is $\varphi$-invariant. Let \(\overline{\varphi}\) be the local self-map induced by \(\varphi\) on \(R/\ker\varphi\). By~\cite[Proposition~5.14]{we2} \(h_{\mathrm{alg}}(\varphi,R)=h_{\mathrm{alg}}(\overline{\varphi},R/\ker\varphi)\), and by assumption $\lambda(\varphi^n)=q(\varphi)^n$ for some $n\in\mathbb{N}$. From Lemma~\ref{forinj} it follows that $\ker\varphi=(0)$.\par
Now let $\mathfrak{q}$ be an $\mathfrak{m}$-primary ideal of $R$. Since $\varphi$ is contracting, by Remark~\ref{alternative} we can choose $t$ large enough so that $\varphi^{nt}(\mathfrak{m})R\subset\mathfrak{q}$. Then there is an exact sequence
\begin{equation}\label{sequence1}
0\rightarrow\frac{\mathfrak{q}}{\varphi^{nt}(\mathfrak{m})R}\rightarrow\frac{R}{\varphi^{nt}(\mathfrak{m})R}\rightarrow R/\mathfrak{q}\rightarrow0.
\end{equation}
From this sequence we obtain
\begin{equation}\label{lengths1}
\ell_R\big(\frac{R}{\varphi^{nt}(\mathfrak{m})R}\big)=\ell_R\big(\frac{\mathfrak{q}}{\varphi^{nt}(\mathfrak{m})R}\big)+\ell_R(R/\mathfrak{q}).
\end{equation}
Now we apply the functor $\eff^n$ to the Sequence~\ref{sequence1} and obtain the exact sequence
$$\eff^n\big(\frac{\mathfrak{q}}{\varphi^{nt}(\mathfrak{m})R}\big)\rightarrow\eff^n\big(\frac{R}{\varphi^{nt}(\mathfrak{m})R}\big)\rightarrow\eff^n\big(R/\mathfrak{q}\big)\rightarrow0.$$
By Proposition~\ref{properties}-\textbf{e} all terms of this sequence are $R$-modules of finite length. From this sequence we obtain
\begin{equation}\label{sequence2}
\ell_R\big(\eff^n(\frac{R}{\varphi^{nt}(\mathfrak{m})R})\big)\leq\ell_R\big(\eff^n(\frac{\mathfrak{q}}{\varphi^{nt}(\mathfrak{m})R})\big)+\ell_R\big(\eff^n(R/\mathfrak{q})\big).
\end{equation}
By Lemma~\ref{forinj}, from the assumption $\lambda(\varphi^n)=q(\varphi)^n$ we conclude $\lambda(\varphi^{nt})=q(\varphi)^{nt}$ for all $t\in\mathbb{N}$. Using this equation and by  Proposition~\ref{properties}-\textbf{d}, we immediately obtain
\begin{eqnarray*} 
\ell_R\big(\eff^n(\frac{R}{\varphi^{nt}(\mathfrak{m})R})\big)&=&\ell_R\big(\frac{R}{\varphi^{nt+n}(\mathfrak{m})R}\big)\\
&=&\lambda(\varphi^{n(t+1)})\nonumber\\
&=&q(\varphi)^{n(t+1)}\nonumber\\
&=&q(\varphi)^{nt}\cdot q(\varphi)^n\nonumber\\
&=&\lambda(\varphi^{nt})\cdot\lambda(\varphi^n).\nonumber
\end{eqnarray*}
Thus, we can re-write Inequality~\ref{sequence2} as
\begin{equation}\label{curious}
\lambda(\varphi^{nt})\cdot\lambda(\varphi^n)\leq\ell_R\big(\eff^n(\frac{\mathfrak{q}}{\varphi^{nt}(\mathfrak{m})R})\big)+\ell_R\big(\eff^n(R/\mathfrak{q})\big).
\end{equation}
On the other hand, by Proposition~\ref{properties}-\textbf{e}
\begin{eqnarray}\label{nostrict}
\ell_R\big(\eff^n(\frac{\mathfrak{q}}{\varphi^{nt}(\mathfrak{m})R})\big)&\leq&\ell_R\big(\frac{\mathfrak{q}}{\varphi^{nt}(\mathfrak{m})R}\big)\cdot\lambda(\varphi^n).\nonumber\\
\ell_R\big(\eff^n(R/\mathfrak{q})\big)&\leq&\ell_R(R/\mathfrak{q})\cdot\lambda(\varphi^n). 
\end{eqnarray}
If Inequality~\ref{nostrict} were strict, then from Inequality~\ref{curious} and above inequalities, we would obtain a strict inequality
$$\lambda(\varphi^{nt})\cdot\lambda(\varphi^n)<\ell_R\big(\frac{\mathfrak{q}}{\varphi^{nt}(\mathfrak{m})R}\big)\cdot\lambda(\varphi^n)+\ell_R(R/\mathfrak{q})\cdot\lambda(\varphi^n),$$
and after dividing by $\lambda(\varphi^n)$:
$$\lambda(\varphi^{nt})<\ell_R\big(\frac{\mathfrak{q}}{\varphi^{nt}(\mathfrak{m})R}\big)+\ell_R(R/\mathfrak{q}).$$
But this inequality would be in contrast to Equation~\ref{lengths1} (note that by Definition~\ref{deflambda} $\lambda(\varphi^{nt})=\ell_R(R/\varphi^{nt}(\mathfrak{m})R)$). This contradiction shows that Inequality~\ref{nostrict} must in fact be an equality, that is, we must have
$$\ell_R\big(\eff^n(R/\mathfrak{q})\big)=\ell_R(R/\mathfrak{q})\cdot\lambda(\varphi^n).$$
By Proposition~\ref{properties}-\textbf{e} and Definition~\ref{deflambda} this means
\begin{equation}\label{fantas}
\ell_R\big(R/\varphi^n(\mathfrak{q})R\big)=\ell_R(R/\mathfrak{q})\cdot\ell_R\big(R/\varphi^n(\mathfrak{m})R\big).
\end{equation}
Since $\mathfrak{q}$ was an arbitrary $\mathfrak{m}$-primary ideal of $R$, by Nagata's Flatness Criterion (Lemma~\ref{Nagata}) Equation~\ref{fantas} tells us that $\varphi^n$ is flat. The implication \textbf{b} $\Rightarrow$ \textbf{a}) of Theorem~\ref{main} applied to $\varphi^n$ then tells us that $R$ is regular, and  the implication \textbf{a} $\Rightarrow$ \textbf{b}) of the same theorem shows that $\varphi$ is flat, as well.
\end{proof}
In~\cite[Definition~3.7]{we2} we proposed the following definition:
\begin{definition}[Hilbert-Kunz multiplicity]\label{HKM}
Let $R$ be a Noetherian local ring of arbitrary characteristic, and let $\varphi$ be a self-map of finite length of $R$. Let $$q(\varphi):=\exp(h_{\mathrm{alg}}(\varphi,R)).$$ Then the Hilbert-Kunz multiplicity of \(R\) with respect to \(\varphi\) is defined as
$$
e_{\mathrm{HK}}(\varphi,R):=\lim_{n\rightarrow\infty}\frac{\lambda(\varphi^n)}{q(\varphi)^n},
$$
provided that the limit exists.
\end{definition}
The following corollary quickly follows from Theorem~\ref{main}.
\begin{corollary}
Let $\varphi$ be a self-map of finite length of a regular local ring $R$. Then $e_{\mathrm{HK}}(\varphi,R)=1$.
\end{corollary}
\bibliographystyle{amsplain}

\end{document}